\documentclass[twocolumn]{autart}  
\usepackage{amsmath,amsfonts,amssymb}
\usepackage{algorithm}
\usepackage{algpseudocode}
\usepackage{array}
\usepackage{xcolor}
\usepackage{flushend}
\usepackage{textcomp}
\usepackage{stfloats}
\usepackage{url}
\usepackage{verbatim}
\usepackage{graphicx}
\usepackage{caption}
\usepackage{subcaption}
\usepackage{booktabs}
\usepackage{mathrsfs}
\usepackage{tikz}
\usetikzlibrary{arrows.meta, decorations.pathreplacing}

\begin{document}

\begin{frontmatter}

\title{Robust Constrained Optimization via Sliding Mode Control} 
\author[a]{Shyam Kamal*}\ead{shyamkamal.eee@iitbhu.ac.in},
\author[a]{Baby Diana*}\ead{babydiana.rs.eee21@itbhu.ac.in},
\author[a]{Sunidhi Pandey}\ead{sunidhipandey.rs.eee20@itbhu.ac.in}, 
\author[a]{Sandip Ghosh}\ead{sghosh.eee@iitbhu.ac.in } 
\author[b,c]{Thach Ngoc Dinh} \ead{ngocthach.dinh@lecnam.net}
\thanks{* Shyam Kamal and Baby Diana contributed significantly and are the primary contributors to this work.}

\address[a]{Dept of Electrical Engineering, IIT(BHU) Varanasi, UP, India} 
\address[b]{Cedric-Lab,Conservatoire National des Arts et M\'etiers, 75141 Paris, France}
\address[c]{Hanse-Wissenschaftskolleg (Institute for Advanced Study), Lehmkuhlenbusch 427753 Delmenhorst, Germany}

\thanks{Supported by the CDRT project, IIT (BHU), Varanasi, under Grant No. R\&D/SA/I-DAPTIIT(BHU)/EE/22-23/08/448.}

\begin{keyword}                           
Constrained optimization, Lagrangian dynamics, Sliding mode control, Finite-time stability (FTS).
\end{keyword}
\maketitle
\thispagestyle{empty}
\begin{abstract}
This paper develops a sliding mode control–based framework for equality-constrained optimization by reformulating the first-order Karush–Kuhn–Tucker (KKT) conditions as a control-affine dynamical system. The optimization variables are treated as states and the Lagrange multipliers as control inputs, with equality constraints defining the sliding manifold. The resulting design guarantees exact constraint enforcement with finite-time convergence, independent of objective convexity, and exhibits robustness to matched disturbances, structural uncertainty and $L_\infty$-bounded measurement noise. To accelerate convergence, a nonsingular terminal sliding mode–based normed gradient flow is introduced, ensuring both finite-time convergence to optimal solutions and constraint satisfaction. Rigorous Lyapunov analysis establishes closed-loop stability and convergence. Numerical studies across diverse benchmark problems demonstrate superior accuracy and robustness over classical continuous-time optimization methods, highlighting effectiveness under disturbances.
\end{abstract}

\end{frontmatter}
\section{Introduction}
 Continuous-time optimization models iterative algorithms as differential equations, arising as infinitesimal-step limits of discrete methods, and provide a powerful framework for analyzing stability and convergence in large-scale convex and nonconvex problems. Gradient flow is the continuous-time analogue of gradient descent, while proximal minimization emerges from backward Euler discretization \cite{parikh2014proximal}. Continuous-time optimization schemes often construct descent directions by combining projected gradient and Gauss–Newton terms to drive feasibility \cite{yamashita1980differential}. These ideas have been extended to inequality-constrained problems by formulating first-order dynamics based on KKT conditions to enforce constraint satisfaction \cite{schropp2000dynamical,zhou2007convergence}.

A substantial body of research has explored continuous-time methods based on Lagrange multipliers, with particular emphasis on primal-dual approaches. The foundational continuous-time method in this context is the primal-dual gradient dynamics (PDGD), first introduced in \cite{arrow1958studies}. Building on this, \cite{qu2018exponential} established the exponential stability of PDGD for strongly convex and smooth objective functions, while subsequent work in \cite{dhingra2018proximal} extended these results to nonsmooth composite optimization problems. Most existing literature emphasizes the continuous-time analysis of known algorithms rather than designing new continuous-time optimization methods. A notable exception is \cite{allibhoy2023control}, which proposes a continuous-time dynamical system using control barrier functions (CBFs) to ensure forward invariance and stability of the feasible set, enabling safe operation in constrained environments. In \cite{feng2020dynamical}, the authors interpret constrained optimization as a perturbed projected dynamical system, linking the escape from local minima to the departure from regions of attraction in the system's behaviour.

In \cite{cerone2025new}, the authors propose a novel continuous-time framework for convex and non-convex equality-constrained optimization, called Controlled Multipliers Optimization (CMO). The approach treats Lagrange multipliers as control inputs in a feedback system, where the constraints act as outputs to be regulated. Two specific control laws—proportional-integral (PI-CMO) and feedback linearization (FL-CMO)—are developed and analyzed. The framework enables the synthesis of new first-order algorithms, with proven convergence in both convex and non-convex settings. 

Continuous-time optimization methods often suffer from asymptotic stability (AS) and approximate constraint satisfaction, especially in non-convex settings. Classical approaches such as gradient flow \cite{schropp2000dynamical}, PDGD \cite{arrow1958studies}, PI-CMO \cite{cerone2025new}, and CBF-based methods \cite{allibhoy2023control} do not ensure exact feasibility in finite time. In contrast, the proposed sliding mode control (SMC) framework guarantees finite-time convergence to the constraint set and robust performance under disturbances. To the best of the authors’ knowledge, such a combination has not been achieved in existing methods beyond convex settings. Although the discontinuous nature of SMC may introduce chattering, practical smoothing strategies are provided to mitigate this effect while preserving the key properties of the proposed approach.

Accordingly, we reformulate the first-order optimality conditions as a dynamical system, following \cite{cerone2025new}, where the optimization variables $x(t)$ are interpreted as system states, the Lagrange multipliers $\lambda(x(t))$ serve as control inputs, and the constraint violations $h(x(t))$ are treated as regulated outputs. This control-theoretic formulation enables robust constraint enforcement with finite-time convergence properties, leveraging sliding mode concepts as developed in \cite{utkin2013sliding,utkin2003variable}. The contributions of this work are as follows:
\begin{enumerate}
    \item \textit{Design of SMC Law:} A novel SMC strategy is proposed for the Lagrange multipliers, guaranteeing finite-time convergence of the constraint violations $h(x(t)) \equiv 0$, irrespective of the convexity of the objective function.
    \item \textit{Robustness Analysis:} The robustness of the proposed approach is analyzed under various classes of disturbances, including matched uncertainties, structural model uncertainties, and measurement noise.
    \item \textit{Finite-time optimality for convex problems:} A nonsingular terminal sliding mode–based normed gradient flow is proposed, ensuring robust finite-time convergence to the global optimum under matched disturbances while preserving feasibility.
    \item \textit{Lyapunov-Based Analysis:} A rigorous Lyapunov-based analysis is presented, establishing exact constraint satisfaction in finite time prior to reaching optimality.
    \item \textit{Numerical Simulation:} The approach is validated on nonconvex optimization, obstacle avoidance, and Shidoku feasibility problems, and on a convex distributed parameter estimation problem, demonstrating superior robustness compared to PDGD, PI-CMO and projected gradient flow (PGF).
\end{enumerate}
The paper is organized as follows. Section II presents the problem formulation and the proposed methodology. Section III develops the main results, including the SMC-based optimization framework along with finite-time convergence and robustness analysis. Section IV illustrates the application results and provides a detailed discussion. Section V concludes the paper. Section VI (Appendix) provides a super-twisting algorithm (STA) design to mitigate chattering effects.
%%%%%%%%%%%%%%%%%%%%%%%%%%%%%%%%%%%%%%%%%%%%%%%%%%%%%
\section{Methodology}
\subsection{Notations}
Let $\mathbb{R}^n$ denote the set of real column vectors of dimension $n$. The $p$-norm $\|x\|_p$ for $p \in \{1,2,\infty\}$ is defined by: $\|x\|_1 = \sum |x_i|$, $\|x\|_2$ (Euclidean norm), and $\|x\|_\infty = \max |x_i|$. The absolute value of a scalar is denoted by $|\cdot|$. For a matrix $A \in \mathbb{R}^{m \times n}$, the range is $\mathrm{range}(A) = \{Ax\}$, and the kernel is $\ker(A) = \{x : Ax = 0\}$. The Jacobian of a function $h : \mathbb{R}^n \to \mathbb{R}^m$ is the matrix $J_h(x) \in \mathbb{R}^{m \times n}$ with entries $\frac{\partial h_i}{\partial x_j}$. The closed ball of radius $\delta > 0$ is $B_\delta(0) = \{x : \|x\|_2 \leq \delta\}$. The sign function is $\operatorname{sign}(x) = x/|x|$ for $x \neq 0$. The set-valued signum function is defined as $\mathrm{SGN}(S(x)_i) = \{+1\}$ if $S(x)_i > 0$, $[-1,1]$ if $S(x)_i = 0$, and $\{-1\}$ if $S(x)_i < 0$. A differential inclusion $\dot{x} \in \mathcal{F}(x)$ becomes an equality $\dot{x} = \mathcal{F}(x)$ when $\mathcal{F}(x)$ is a singleton. A function $\phi \in \mathcal{C}^1(\mathbb{R}^n,\mathbb{R})$ is $\mu$-strongly convex $(\mu>0)$ if $\phi(y) \ge \phi(x) + \nabla \phi(x)^\top (y-x) + \frac{\mu}{2}\|y-x\|^2,
\quad \forall x,y \in \mathbb{R}^n$.

\subsection{Preliminaries}
Consider an optimization problem with equality constraints:
\begin{equation}
\min_{x \in \mathbb{R}^n} \phi(x) \quad \text{subject to} \quad h(x) = 0,
\label{eq:op}
\end{equation}
where \( \phi: \mathbb{R}^n \rightarrow \mathbb{R} \) and \( h: \mathbb{R}^n \rightarrow \mathbb{R}^m \) are continuously differentiable (possibly non-convex) functions. The Lagrangian is defined as:
% \begin{equation}
$\mathcal{L}(x, \lambda) = \phi(x) + \lambda^\top h(x)$,
% \label{eq:lang_fun}
% \end{equation}
with Lagrange multipliers \( \lambda \in \mathbb{R}^m \).

\begin{lem}
\label{lem: FONC_SONC}
Let \( x^* \) be a local minimum of \( \phi \), where \( h \) is continuously differentiable and the constraint qualification \( \text{rank}(J_h(x^*)) = m \) holds. Then, there exists \( \lambda^* \in \mathbb{R}^m \) such that:

\begin{itemize}
    \item \textit{First-order necessary condition (FONC):}
    \begin{equation}
        \nabla_x \mathcal{L}(x^*, \lambda^*) = \nabla \phi(x^*) + J_h(x^*)^\top \lambda^* = 0, ~~h(x^*) = 0, 
    \label{eq:fonc}
    \end{equation} 
where \( J_h(x) \in \mathbb{R}^{m \times n} \) is the Jacobian of the constraint function \( h(x) \).

    \item \textit{Second-order necessary condition (SONC):} \\
    For non-convex problems, the first-order conditions identify critical points that may be minima, maxima, or saddle points. The second-order condition uses the Hessian \( \nabla^2_{xx} \mathcal{L}(x^*, \lambda^*) \) to further classify the point:
    % \begin{equation}
$    d^\top \nabla^2_{xx} \mathcal{L}(x^*, \lambda^*) d \geq 0 \quad \forall d \in \ker(J_h(x^*))$.
    % \label{eq:sonc}
    % \end{equation}
\end{itemize}

\end{lem}
The FONC and SONC for equality-constrained optimization problems arise naturally from the KKT framework, which generalizes the method of Lagrange multipliers \cite{luenberger1984linear}.
%%%%%%%%%%%%%%%%%%%%%%%%%%%%%%%%%%%%%%%%%%%%%%%%%%%%%%%%%%%%%%%%%%%%%%%%%%%

\subsection{Problem Formulation}
Consider the constrained optimization problem \eqref{eq:op}, motivated by a dynamical systems interpretation of the stationarity conditions \eqref{eq:fonc}:
\begin{equation}
\mathcal{H}: 
\begin{cases}
\dot{x}(t) = -\nabla \phi(x(t)) - J_h(x(t))^\top \lambda(x(t)), \label{eq: gf dynamics} \\
y(t) = h(x(t)),
\end{cases}
\end{equation}
where $\lambda(x(t))$ is treated as a control input, and $y(t)$ serves as the regulated output i.e. the constraint surface. This control-oriented perspective motivates the design of a feedback controller for $\lambda(x(t))$ that drives the system toward satisfaction of the KKT conditions. An equilibrium $(x^*, \lambda^*)$ of the system $\mathcal{H}$ qualifies as a stationary point if and only if the constraint condition $h(x^*) = 0$ is satisfied, as established in \cite{cerone2025new}.\\
\textbf{Problem statement:} 
\begin{enumerate}
    \item Design a feedback control law $\lambda$ for the system $\mathcal{H}$ such that the state $x(t)$ converges to the optimal solution $x^*$ asymptotically, and the constraint violation $y(t)$ vanishes in finite time. Specifically, ensure that 
% \begin{equation}
$\lim_{t \to \infty} x(t) = x^*, \lim_{t \to T} ~y(t) = 0$,
% \end{equation}
for some $T \in [0, \infty]$.
\item  Design an NTSMC control law $\lambda$ to guarantee finite-time reaching of the sliding manifold, while the gradient term in \eqref{eq: gf dynamics} is replaced by a normed gradient flow to ensure finite-time convergence to optimum point under a $\mu$-strongly convex objective $\phi$.
\end{enumerate}
To address this, we develop an SMC strategy in the following sections.\footnote{For clarity, dependence on $t$ and $x$ is omitted when clear; the sliding variable and manifold are both denoted by $s(x)$.}

%%%%%%%%%%%%%%%%%%%%%%%%%%%%%%%%%%%%%%%%%%%%%%%%%%%%%%%%%%%%%%
\section{Main result}
\subsection{Proposed Sliding Mode Controller}
Consider the dynamical system \eqref{eq: gf dynamics}. The sliding surface is constructed based on the output to enforce the constraint, and is given by:
\begin{equation}\label{eq: SF}
S(x) = \{x \in \mathbb{R}^n \mid h(x) = 0\}.
\end{equation}
The control objective is to ensure sliding variable i.e. \( S(x)=h(x)\) to 0 in finite time, and the closed-loop system is driven toward a stationary point $(x^*, \lambda^*)$ satisfying the first-order necessary conditions \eqref{eq:fonc}.
To enforce the constraint $h(x) \to 0$ in $0<T<\infty$, we propose the following discontinuous feedback law for 
% \begin{equation}
$\lambda(x) = -\left(J_h(x) J_h(x)^\top\right)^{-1} \big [J_h(x) \nabla \phi(x) - K \mathrm{sgn}(h(x))\big]$,
% \label{eq:smc}
% \end{equation}
where $K \in \mathbb{R}^{m \times m}$ is a diagonal positive-definite gain matrix, and $\mathrm{sgn}(\cdot)$ denotes the elementwise sign function.
% The first term ensures alignment with the gradient descent direction projected onto the constraint manifold, while the second term induces a sliding mode that robustly drives the constraint violation $h(x(t))$ to zero in finite time.% Under this control law, the sliding mode dynamics ensure convergence to the constraint surface, while the projected gradient component ensures descent along the objective within the feasible region.
The control input \( \lambda(x) \) consists of: (I) an \textit{equivalent control} \( \lambda_{eq}(x) \) to maintain \( S(x) = 0 \) and (II) a \textit{switching term} \( \lambda_{sw}(x) \) to drive \( S(x) \rightarrow 0 \). Thus,
% \begin{equation}
$\lambda(x) = \lambda_{eq}(x) + \lambda_{sw}(x)$. 
% \label{eq: control}
% \end{equation}
The sliding mode control design is based on Utkin's principle of equivalent control \cite{utkin2003variable}, which ensures the existence of sliding motion when the reachability condition is satisfied, as established in \cite{utkin2013sliding}.
%%%%%%%%%%%%%%%%%%%%%%%%%%**********************%%%%%%%%%%%%%%%%
\begin{thm} Consider the controlled dynamical system $\mathcal{H}$ defined in \eqref{eq: gf dynamics} with sliding manifold \eqref{eq: SF} defined by the equality constraints. A sliding mode on $S(x)$ exists if the following conditions hold:

\begin{enumerate}
\item \textit{Invariance Condition}: There exists an equivalent control input $\lambda_{eq}(x) \in \mathbb{R}^m$ that renders $S(x)$ invariant,
$J_h(x)\dot{x} \big|_{\lambda=\lambda_{eq}(x)} = 0,$
which yields
% \begin{equation}\label{eq:invariance_condition}
  $ \lambda_{eq}(x) = -(J_h(x)J_h(x)^\top)^{-1}J_h(x)\nabla \phi(x)$. 
% \end{equation}
\item \textit{Reachability Condition}: The switching control component is given by, 
% \begin{equation} \label{eq:switching control}
    $\lambda_{sw}(x) = (J_h(x)J_h(x)^\top)^{-1} K  \mathrm{sgn}(h(x))$
% \end{equation}
where $K = \operatorname{diag}(k_1, \dots, k_m)$,
guarantees finite-time convergence to $S(x)$ with reaching time bounded by, $t_r \leq \frac{\sqrt{2||h(0)||_2}}{k_{\min}}$, where $k_{\min} = \min_i(k_i), i=1,2 \dots m$.
\end{enumerate} 
\end{thm}
\begin{pf}
We prove both conditions of Theorem separately.\\
\noindent\textit{Part 1 (Invariance Condition):}\\
To ensure that the system trajectory remains on the sliding manifold defined by \eqref{eq: SF}, its time derivative must satisfy
% \begin{equation*}
    $\frac{d}{dt} h(x(t)) = J_h(x)\dot{x} = 0$. 
    % \label{eq:slide_condition}
% \end{equation*}
Substituting the system dynamics $\dot{x} = -\nabla \phi(x) - J_h(x)^\top \lambda_{\mathrm{eq}}(x)$ into above sliding condition, we obtain
% \begin{equation*}
$J_h(x)\left(-\nabla \phi(x) - J_h(x)^\top \lambda_{\mathrm{eq}}(x)\right) = 0$, 
  % \label{eq:slide_dynamics}
% \end{equation*}
which simplifies to the linear system
% \begin{equation*}
    $ J_h(x) J_h(x)^\top \lambda_{\mathrm{eq}}(x) = -J_h(x) \nabla \phi(x)$.
% \label{eq:lambda_eq_linear}
% \end{equation*}
Under the regularity condition $\mathrm{rank}(J_h(x)) = m$, the matrix $J_h(x) J_h(x)^\top$ is invertible, yielding the equivalent control input
% \begin{equation*}
  $\lambda_{\mathrm{eq}}(x) = -\left(J_h(x) J_h(x)^\top\right)^{-1} J_h(x) \nabla \phi(x)$. 
% \label{eq:lambda_eq}
% \end{equation*}
% To maintain invariance of the sliding manifold \eqref{eq: SF}, the sliding condition
% \[
% \dot{h}(x)=J_h(x)\dot{x}=0
% \]
% must hold. Substituting the system dynamics $\dot{x}=-\nabla \phi(x)-J_h(x)^\top\lambda_{\mathrm{eq}}$ yields
% \[
% J_h(x)\nabla \phi(x)+J_h(x)J_h(x)^\top\lambda_{\mathrm{eq}}=0.
% \]
% Since $\operatorname{rank}(J_h(x))=m$, the matrix $J_h(x)J_h(x)^\top$ is invertible, and the equivalent Lagrange multiplier is uniquely determined as
% \[
% \lambda_{\mathrm{eq}}
% =
% -\big(J_h(x)J_h(x)^\top\big)^{-1}J_h(x)\nabla \phi(x).
% \]

\noindent\textit{Part 2 (Reachability Condition): }\\
To establish finite-time convergence to the sliding manifold \eqref{eq: SF}, consider the Lyapunov function candidate
$V(x) = \frac{1}{2}\|h(x)\|_2^2.$
Taking its time derivative along system trajectories yields
% \begin{equation*}
    $\dot{V}(x) = h(x)^\top \dot{h}(x) = h(x)^\top J_h(x) \dot{x}$.
% \end{equation*}
Substituting the closed-loop dynamics $\dot{x} = -\nabla \phi(x) - J_h(x)^\top (\lambda_{\mathrm{eq}}(x) + \lambda_{\mathrm{sw}}(x))$ and using $\lambda_{\mathrm{eq}}(x) = -(J_h(x)J_h(x)^\top)^{-1}J_h(x)\nabla \phi(x)$ with $\lambda_{\mathrm{sw}}(x)= (J_h(x)J_h(x)^\top)^{-1} K  \mathrm{sgn}(h(x))$, we obtain
% \begin{equation*}
    $\dot{V}(x) = -K h(x)^\top \mathrm{sgn}(h(x)) = -K \|h(x)\|_1$,
% \end{equation*}
where $K = \operatorname{diag}(k_1, \dots, k_m)$. Since $\|h(x)\|_1 \geq \|h(x)\|_2$, it follows that
% \begin{equation*}
    $\dot{V}(x) \leq -k_{\min} \|h(x)\|_2 = -k_{\min} \sqrt{2V(x)}$,
% \end{equation*}
where $k_{\min} = \min_i(k_i)$, $i = 1, \dots, m$. Solving the differential inequality $\dot{V}(x) \leq -k_{\min} \sqrt{2V(x)}$ by separation of variables yields
% \begin{equation*}
    $\int_{V(0)}^0 \frac{dV}{\sqrt{V}} \leq -k_{\min} \sqrt{2} \int_0^{t_r} dt \quad \Rightarrow \quad 2\sqrt{V(0)} \leq k_{\min} \sqrt{2} t_r$,
% \end{equation*}
which implies the reaching time is bounded by
$ t_r \leq \frac{\sqrt{2V(0)}}{k_{\min}} = \frac{\sqrt{2} \|h(0)\|_2}{k_{\min}}.$
Hence, the system trajectory reaches the sliding manifold $S(x)$ in finite time.
\end{pf}
\begin{rem}
The complete control law $\lambda = \lambda_{eq}(x) + \lambda_{sw}(x)$ yields:
\begin{equation} \label{eq:closed_loop}
\dot{x} = \underbrace{-P(x)\nabla \phi(x)}_{\text{Feasible descent}} - \underbrace{J_h(x)^\top K  \mathrm{sgn}(h(x))}_{\text{Constraint enforcement}}
\end{equation}
where $P(x) = I - J_h(x)^\top (J_h(x) J_h(x)^\top)^{-1} J_h(x)$ is the orthogonal projector onto $\text{ker}(J_h(x))$  ensuring $h(x) \equiv 0$ for $t \geq t_r$. The closed loop dynamics \eqref{eq:closed_loop} exhibits two phase behavior i.e the switching term drives \( h(x) \to 0 \) in finite time (reaching phase), after which the reduced dynamics \( \dot{x} = -P(x)\nabla \phi(x) \) govern feasible optimization (sliding phase). 
\end{rem}
% The closed-loop dynamics \eqref{eq:closed_loop} exhibits two-phase behavior (as discussed in Table \ref{tab:smc_components} ):

% \begin{table}[h!]
% \centering
% \footnotesize
% \begin{tabular}{@{}p{0.4\columnwidth}p{0.5\columnwidth}@{}}
% \toprule
% \textbf{Component} & \textbf{Function} \\ \midrule
% $\lambda_{\text{eq}}$ & Maintains $h(x)=0$ on sliding manifold \\
% $\lambda_{\text{sw}}$ & Ensures finite-time convergence to manifold \\
% $P(x)\nabla \phi(x)$ & Projects gradient onto tangent space \\
% $J_h(x)^\top$ & Constraint Jacobian mapping \\
% \bottomrule
% \end{tabular}
% \caption{ Key components of proposed sliding mode control}
% \label{tab:smc_components}
% \end{table}

\begin{rem}
Upon reaching the sliding manifold $S$ in \eqref{eq: SF} at time $t_r$, the dynamics reduce to the projected gradient flow $\dot{x}(t) = -P(x(t))\nabla \phi(x(t))$ for $t \ge t_r$, which enforces $h(x(t)) \equiv 0$ thereafter. Under this flow, trajectories converge asymptotically to the set $\mathcal{E} = \{ x^* \in S(x) \mid P(x^*)\nabla \phi(x^*) = 0 \}$ of constrained critical points, where $\nabla \phi(x^*) \in \mathrm{range}(J_h(x^*)^\top)$, satisfying the KKT conditions. Hence, the SMC framework guarantees finite-time constraint enforcement and asymptotic convergence to KKT points. However, since the descent direction $-P(x)\nabla \phi(x)$ is local, global optimality is not guaranteed in general nonconvex settings, and the attained solution may depend on the initialization, as is typical for first-order methods.
\end{rem}

\begin{rem} 
While $\operatorname{sgn}(h)$ is used in theory, practical implementations adopt $\operatorname{sat}(h/\epsilon)$ to reduce chattering. Under discretization with sampling time $\tau$, a perturbation $\delta\lambda(t)$ order of $\mathcal{O}(\tau K)$ induces chattering of amplitude $\mathcal{O}(\tau K \|J_h^\dagger\|)$, confined to the normal subspace without affecting the descent term $-P(x)\nabla \phi(x)$. Thus, optimization dynamics on the constraint manifold are preserved. Choosing $\epsilon \ge \tau K$ ensures continuous control within the boundary layer, at the cost of a steady-state error $\mathcal{O}(\epsilon)$.
\end{rem}

\textit{Example 1:}
Consider the univariate constrained optimization problem:
\begin{equation}
\min_{x \in \mathbb{R}} \frac{1}{2} w x^2 ~~ \text{s.t.} ~~ x = 0,
\end{equation}
where $w > 0$. To achieve finite-time convergence, we design an SMC law based on the sliding variable $S(x) = x$. The equivalent control on the manifold $S(x) = 0$ is $\lambda_{\text{eq}}(x) = -w x$, obtained by enforcing $\dot{s} = 0$, and the switching control is $\lambda_{\text{sw}}(x) = -K\mathrm{sgn}(x)$, with $K > 0$. The resulting closed-loop system becomes:
% \begin{equation}
$\dot{x} = -K\mathrm{sgn}(x)$.
% \end{equation}
Using $V=\tfrac{1}{2}x^2$, the SMC dynamics satisfy $\dot{V}=-K|x|=-\sqrt{2}K V^{1/2}$, which guarantees finite-time convergence with settling-time bound $t_s \le \sqrt{2V(0)}/K$. Table~\ref{tab:comparison} compares existing methods with the proposed SMC, highlighting that while all methods achieve asymptotic convergence to the local KKT point, the proposed method ensures finite-time reachability of the feasible set at $t_s$.
\begin{table}[t]
\centering
\caption{Convergence Comparison}
\label{tab:comparison}
\scriptsize
\setlength{\tabcolsep}{1.2pt}
\renewcommand{\arraystretch}{0.9}
\begin{tabular}{lcccc}
\toprule
Method & Feasibility & Local KKT Conv. & Rate & Params. \\
\midrule
PDGD \cite{arrow1958studies}  
& Asym. & AS 
& \( \|x\|\le ce^{-\lambda t}\|x_0\| \) 
& \(K_i\) \\

PI-CMO \cite{cerone2025new}  
& Asym. & AS 
& \( \|x\|\le ce^{-\alpha t}\|x_0\| \) 
& \(K_p,K_i\) \\

SMC (Prop.)  
& Finite & AS 
& \( t_s \le |x_0|/K \) 
& \(K\) \\
\bottomrule
\end{tabular}
\end{table}
With the ideal control \( \dot{h}=-K\operatorname{sgn}(h) \), finite-time convergence is achieved with \( T_{\text{reach}}=|h(0)|/K \); replacing \( \operatorname{sgn}(h) \) by \( \operatorname{sat}(h/\epsilon) \) yields convergence to a boundary layer \( \|h\|=\epsilon \) with exponential decay, where \(K\) and \(\epsilon\) trade off speed, accuracy, and smoothness, while higher-order methods such as the Super-Twisting Algorithm (Appendix) achieve exact finite-time convergence to feasible set.

\subsection{Robustness of Sliding Mode Optimization Controller}
The sliding mode optimization controller possesses inherent robustness, which can be rigorously analyzed within the framework of differential inclusions and set-valued mappings. Theorem 6 formalizes these robustness properties by considering various perturbation scenarios, including matched, structural uncertainty and $L_\infty$-bounded measurement noise.
\begin{thm}
    Consider the nominal system \eqref{eq: gf dynamics} augmented with a Filippov sliding mode control law:
\begin{equation} \label{eq: FSmc}
    \lambda \in - (J_hJ_h^\top)^{-1} J_h \nabla \phi(x) + K(J_hJ_h^\top)^{-1} \mathrm{SGN}(S(x)),
\end{equation}
Under this control law, the system exhibits the following robustness properties:
\begin{enumerate}
    \item \textit{Matched Disturbances:}
    If the disturbance $\xi(t,x) \in \text{range}(J_h^\top) = \ker(J_h)^\perp$, then the sliding manifold $S(x) = \{x \in \mathbb{R}^n \mid h(x) = 0\}$
remains invariant for the perturbed system:
\begin{equation} \label{eq: perturbed system}
    \dot{x} = -\nabla \phi(x) - J_h^\top \lambda + \xi(t,x).
\end{equation}
where $\xi(t, x) = J_h^\top \eta(t, x)$, where $\|\eta(t, x)\| \leq \bar{\eta}$.
    \item \textit{Structured Uncertainty: }
 If the control matrix is subject to bounded time-varying perturbations $\Delta J_h (t)$, and the modified matrix preserves full rank:
% \begin{equation} \label{eq: Constraint unc}
  $ \inf_{t \geq 0} \sigma_{\min}\left[(J_h + \Delta J_h(t))(J_h + \Delta J_h(t))^\top\right] \geq \alpha > 0$, 
% \end{equation}
Then, the sliding phase and convergence properties are maintained.
    \item \textit{Measurement Noise:}
 If the sliding variable is corrupted by $\mathcal{L}_\infty$- bounded noise $\tilde \eta(t)$ with $\|\tilde \eta(t)\| \leq \delta$, the system evolves as:
% \begin{equation}
    $\dot{x} = -\nabla \phi(x) - J_h^\top \lambda(S(x)_{meas})$,
% \end{equation}
and the state remains within a boundary layer around the manifold, i.e., $\|S(x)\| \leq \Delta(\delta, K)$, where $\Delta$ is a function of the noise level and control gain.
\end{enumerate}
\end{thm} 
\begin{pf}
We provide a detailed analysis for all considered perturbation scenarios to comprehensively establish the robustness properties of the sliding mode optimization controller.
\begin{enumerate}
    \item \textit{Matched Disturbance Rejection: }
Consider the perturbed system \eqref{eq: perturbed system} with matched disturbances of the form $\xi(t, x) = J_h^\top \eta(t, x)$, where $\|\eta(t, x)\| \leq \bar{\eta}$. Let the sliding surface be defined as in \eqref{eq: SF}. Substituting the control input \eqref{eq: FSmc} into the system yields the sliding variable dynamics:
% \begin{equation*}\label{eq:Sdot}
    $\dot{S(x)} = -K\,\mathrm{SGN}(S(x)) + J_h J_h^\top \eta(t, x)$.
% \end{equation*}
Define the Lyapunov candidate $V(S(x)) = \frac{1}{2} \|S(x)\|_2^2$. Taking its time derivative gives:
% \begin{equation*}
   $ \dot{V} = S(x)^\top \dot{S(x)} = -K \|S(x)\|_1 + S(x)^\top J_h J_h^\top \eta(t, x)$.
% \end{equation*}
Using Cauchy-Schwarz and the upper bound on the disturbance, we get:
\begin{align*}
    \dot{V} &\leq -K \|S(x)\|_2 + \|S(x)\|_2 \|J_h J_h^\top \eta(t,x)\|_2 \nonumber \\
           &\leq -K \|S(x)\|_2 + \sigma_{\max}(J_h J_h^\top) \bar{\eta} \|S\|_2 \nonumber \\
           &= -\epsilon \|S(x)\|_2 = -\epsilon \sqrt{2V},
\end{align*}
where $\epsilon = K - \sigma_{\max}(J_h J_h^\top) \bar{\eta}$.
Finite-time convergence to the sliding manifold is guaranteed if $\epsilon > 0$. Solving the inequality $\dot{V} \leq -\epsilon \sqrt{2V}$ by separation of variables yields:
$t_r \leq \frac{\|S(0)\|_2}{\epsilon}.$
Hence, the system reaches the sliding surface in finite time under bounded matched disturbances.
    \item \textit{Structured Uncertainty: }
Consider the perturbed system \eqref{eq: perturbed system} subject to structured uncertainty in the constraint Jacobian, modeled as $\tilde{J}_h = J_h + \Delta J_h$. Define the sliding surface as $S = \tilde{J}_h x$. Under the proposed control law, its time derivative is given by:
% \begin{equation*}
    $\dot{S(x)} = -K\,\mathrm{SGN}(S(x)) + \tilde{J}_h \tilde{J}_h^\top \eta(t,x)$.
% \end{equation*}
Choose the Lyapunov function $V = \frac{1}{2} \|S\|_2^2$. Its derivative along the system trajectories is:
\begin{align*}
    \dot{V} &= S^\top \dot{S} \leq -K \|S\|_1 + \|S\|_2 \| \tilde{J}_h \tilde{J}_h^\top \eta(t,x) \|_2 \nonumber \\
    &\leq -K \|S\|_2 + \sigma_{\max}(\tilde{J}_h \tilde{J}_h^\top)\bar{\eta} \|S\|_2 \nonumber \\
    &= -\tilde{\epsilon} \|S\|_2 = -\tilde{\epsilon} \sqrt{2V},
\end{align*}
where $\tilde{\epsilon} = K - \sigma_{\max}(\tilde{J}_h \tilde{J}_h^\top)\bar{\eta} > 0$. Therefore, the trajectories converge to the sliding surface in finite time. Solving the differential inequality yields the upper bound on the reaching time:
$ t_r \leq \frac{\|S(0)\|_2}{\tilde{\epsilon}}.$
    \item \textit{Noisy Measurements: }
Assume the measured sliding surface is corrupted by noise, given as $S_{\text{meas}} = S + \tilde{\eta}(t)$, where $\|\tilde{\eta}(t)\| \leq \delta$, and the matched disturbance satisfies $\|\eta(t,x)\| \leq \bar{\eta}$. The sliding dynamics then evolve as:
% \begin{equation*}
    $\dot{S}_{\text{meas}} = -K\,\mathrm{SGN}(S_{\text{meas}}) + J_h J_h^\top \eta(t,x)$.
% \end{equation*}
Let $V = \frac{1}{2} \|S\|_2^2$ be the Lyapunov candidate. Using the property $\mathrm{SGN}(S + \tilde{\eta}) \subset \mathrm{SGN}(S) + \mathscr{B}_\delta(0)$, its derivative satisfies:
\begin{align*}
    \dot{V} &\leq -K \|S\|_2 + \|S\|_2 \left( K \delta + \sigma_{\max}(J_h J_h^\top) \bar{\eta} \right) \nonumber \\
    &= -\left( K - K \delta - \sigma_{\max}(J_h J_h^\top) \bar{\eta} \right) \|S\|_2.
\end{align*}
Thus, the system is uniformly ultimately bounded, and the ultimate bound is given by:
%\begin{equation*}
   $ \|S(x)\|_2 \leq \frac{K\delta + \sigma_{\max}(J_h J_h^\top)\bar{\eta}}{K}$.
%\end{equation*}
\end{enumerate}
\end{pf}
The closed-loop system enjoys strong robustness properties, including a singular perturbation structure with fast sliding and slow projected dynamics on the constraint manifold, input-to-state stability with respect to disturbances, and graph-closedness, ensuring robustness under approximate or perturbed implementations.
%%%%%%%%%%%%%%%%%%%%%%%%%%%%%%%%%%%%%%%%%%%%%%%%%%%%%%%%%%%%%%%%%%%%%%%%%%%%%%%%%%%%%%%%%%%%%%%%%%%%%
\subsection{Finite-Time Nonsingular TSMC Design for Global Optimality in Convex Optimization}
Conventional sliding mode control $S = h(x)$ achieves finite-time constraint satisfaction via $\dot{S}=-K\operatorname{sgn}(S)$, but slows near $h=0$ unless large gains are used. To overcome this, a nonsingular terminal sliding mode (NTSM) \cite{feng2013nonsingular} with $z_1 = h(x)$, $z_2 = \dot h(x)$, and  $S = z_1 + \beta |z_2|^{\gamma} \mathrm{sgn}(z_2), \;\beta>0,\,1<\gamma<2,$ enforces fast, state-dependent finite-time convergence. In the primal dynamics, only the normalized gradient of $f$ as in \cite{diana2025finite} is used rather than $\nabla \mathcal{L} = \nabla \phi + J_h^\top \lambda$, since SMC guarantees $h(x)=0$ in finite time; once the manifold is reached, $\lambda^\top h = 0$ and the optimization reduces to minimizing $\phi$ on the constraint manifold. This decouples the roles: SMC robustly enforces feasibility, while the gradient flow drives the state to the optimum without redundancy or interference from $\lambda$-terms, under bounded disturbance $\|\xi(t)\| \le D$.
\begin{equation} \label {eq:ntsm_sys}
    \mathcal{H}_1
    \begin{cases}
\dot{x} = -\frac{\nabla \phi(x)}{\|\nabla \phi(x)\|^{p}} - J_h(x)^\top \lambda + \xi(t), \quad \dot{\lambda} = u,\\
y=h(x) 
    \end{cases}
\end{equation}
where $p=\frac{b-2}{b-1}, b>2$, $\lambda\in\mathbb{R}^m$ is the dual state, $u\in\mathbb{R}^m$ is the control input. Differentiating $y=h(x)$ twice along \eqref{eq:ntsm_sys} yields $\ddot{y}
= a(x,\lambda,\xi) - J_h(x)J_h(x)^\top u$, establishing a vector relative degree $\{2,\dots,2\}$. The drift term is given by $a(x,\lambda,\xi) = H_h(x)[\dot{x},\dot{x}]
+ J_h(x)J_{f_p}(x)\dot{x}- J_h(x)\dot{J}_h(x)^\top \lambda + J_h(x)\dot{\xi}(t)$, which is bounded under the stated assumptions.

\noindent\textbf{Assumptions.}
(i) $\phi\in\mathcal{C}^2$ is $\mu$-strongly convex with locally Lipschitz gradient;  
(ii) $h\in\mathcal{C}^2$ and $J_h(x)$ has full row rank in a neighborhood of the optimizer;  
(iii) $\xi(t)$ and $\dot{\xi}(t)$ are bounded;  
(iv) all closed-loop signals remain bounded prior to sliding.
\begin{thm}
Consider system \eqref{eq:ntsm_sys} under Assumptions~1--4. Let the control input be designed as
\begin{align}
u &= \left( J_h(x) J_h(x)^{\top} \right)^{-1} \Big[
\bar{a}(x,\lambda)+ \eta \mathrm{sgn}(S) \nonumber\\
&+ \frac{\beta}{\gamma} |z_2|^{1-\gamma} \odot
\big( K_1 S + K_2 |S|^{\rho} \mathrm{sgn}(S) + z_2 \big)
\Big], \label{eq: ntsm control law}
\end{align}
where $K_1,K_2 \succ 0$ are diagonal gain matrices, $0<\rho<1$, $\beta>0$, $1<\gamma<2$, and $\eta>D$. The nonsingular terminal sliding manifold is defined as $S = z_1 + \frac{1}{\beta} |z_2|^{\gamma} \mathrm{sgn}(z_2), \quad z_1=h(x),\; z_2=\dot{h}(x)$. Then, the closed-loop system satisfies:
\begin{enumerate}
\item the sliding variable $S(x)$ converges to zero in finite time $T_1$;
\item the constraints $h(x)=0$ and $\dot{h}(x)=0$ are satisfied in finite time
$T_1+T_2$;
\item the state $x(t)$ converges to the unique KKT point $x^\star$ in finite time
$T_1+T_2+T_3$.
\end{enumerate}
Then, the constrained convex optimization problem is solved exactly in finite
time.
\end{thm}
\begin{pf}
The proof is divided into three steps.
\begin{enumerate}
    \item \textit{Finite-Time Reaching of the Sliding Manifold: }
    Substituting the control law \eqref{eq: ntsm control law} into the $S$-dynamics yields $\dot{S} = -K_1 S - K_2 |S|^{\rho} \mathrm{sgn}(S)-\eta \mathrm{sgn}(S) + \xi(t)$, where $\xi(t)$ is a matched disturbance bounded by $\|\xi(t)\|\leq D$. Consider the Lyapunov function $V_S=\frac{1}{2}S^{\top}S$. Its derivative satisfies
\[
\dot{V}_S
\leq
-\lambda_{\min}(K_1)\|S\|^2
-\lambda_{\min}(K_2)\|S\|^{1+\rho},
\]
since $\eta>D$. $S$ reaches the origin in finite time $T_1$.
    \item  \textit{ Finite-Time Constraint Satisfaction: }
    On the sliding manifold $S=0$, the nonsingular terminal constraint implies $z_1= -\frac{1}{\beta}|z_2|^{\gamma}\mathrm{sgn}(z_2)$. Using $z_2=\dot{z}_1$, the reduced-order dynamics become $\dot{z}_1 = -\beta^{1/\gamma}|z_1|^{1/\gamma}\mathrm{sgn}(z_1)$, which is homogeneous of negative degree. Therefore, $z_1=h(x)$ and $z_2=\dot{h}(x)$ converge to zero in finite time $T_2$.
\item \textit{ Finite-Time Optimality on the Constraint Manifold: }
For all $t\geq T_1+T_2$, the system evolves on $h(x)=0$ with $\dot{h}(x)=0$.
The equivalent dynamics reduce to $\dot{x}= -P(x)\frac{\nabla \phi(x)}{\|\nabla \phi(x)\|^{p}} + P(x)\xi(t)$, where $P(x) = I_n - J_h(x)^{\top}\big(J_h(x)J_h(x)^{\top}\big)^{-1}J_h(x)$ is the projection onto the tangent space of the constraint manifold. Consider the Lyapunov function $V=\frac{1}{2}\|x-x^\star\|^2$. Using $\mu$ strong convexity of $\phi$ and smoothness of $\nabla \phi$, one obtains:
\[
\dot{V}
\leq
-c_1 V^{\frac{2-p}{2}} + c_2 V^{\frac{1}{2}},
\quad 0<p<1,
\]
for positive constants $c_1,c_2$. In the absence of disturbances,
$\dot{V}\leq -\tilde{c}V^{\frac{2-p}{2}}$ near the equilibrium, which guarantees
finite-time convergence to $x^\star$. Hence, the KKT point $(x^\star,\lambda^\star)$ is finite-time stable.
\end{enumerate}
The NTSM-based controller ensures finite-time convergence in three phases: the sliding variable reaches the manifold in $T_1 \le \frac{2}{\lambda_{\min}(K_1)} \ln\Big( 1 + \frac{\lambda_{\min}(K_1) \|s(0)\|^{1-\rho}}{\lambda_{\min}(K_2)(1-\rho)} \Big)$, constraints converge on the manifold in $$T_2 = \frac{\gamma}{\gamma-1} \frac{\max_i |h_i(x(0))|^{(\gamma-1)/\gamma}}{\beta^{1/\gamma}},$$ and the zero dynamics reach the optimum in  $T_3 \le \frac{2}{\tilde{c}(2-p)} \left( \frac{1}{2}\|x(T_2)-x^*\|^2 \right)^{(2-p)/2}, \quad \tilde{c} = \mu/L^p$. Thus, the total convergence time is bounded by \(T_{\rm total} \le T_1 + T_2 + T_3\), with \(T_1\) logarithmic and \(T_2, T_3\) power-law dependent on initial errors.
\end{pf}

\section{Discussion}
In this section, we apply the proposed method to various applications and compare it with existing approaches. For practical implementation, parameter tuning is essential; thus, we provide a concise guideline. First, estimate the disturbance bound $\bar{\eta}$ and noise level $\delta$. Then select $K_{\min} = 2\,\sigma_{\max}(J_h J_h^\top)\bar{\eta}$, and choose $\epsilon = \tau K_{\min} + \delta$ to suppress chattering. Verify the Jacobian condition $\underline{\sigma} > \delta_J \|J_h\|$, and implement the smoothed control $\lambda = -(J_h J_h^\top)^{-1}[J_h \nabla \phi - K \operatorname{sat}(S/\epsilon)]$.

\subsection{Non-convex Quadratic Optimization Problem}
We consider a quadratic nonconvex optimization problem with a linear equality constraint, as studied in \cite{cerone2025new}:
\begin{align}
\min_{x \in \mathbb{R}^n} \quad & \frac{1}{2} x^\top W x \nonumber \\
\text{subject to} \quad & Cx + d = 0,
\end{align}
where $W = \mathrm{diag}\{1, -1\}$, $C = [0; 2]$, and $d = 0$ with invertible $C^TC$. We analyze the perturbed system dynamics given in \eqref{eq: perturbed system}, subject to an external disturbance $\xi(t) = 2\sin(t)$, which satisfies the boundedness condition specified in Theorem 6.  As illustrated in Fig. \ref {fig:comp_disturbance}, the proposed SMC strategy \eqref{eq: FSmc} ensures robust disturbance rejection and enforces convergence to the sliding manifold at $T_C = 0.495~\text{sec}$. 
% \begin{figure}[h!]
%     \centering
%     \includegraphics[width=0.8\linewidth]{opt_comp_disturbance.eps}
%     \caption{This figure shows validation of Theorem 2 by comparing the proposed SMC-based constrained optimization method with PI-CMO \cite{cerone2025new} and PDGD under disturbance. While both PI-CMO and PDGD diverge in the presence of disturbances, the SMC-based approach achieves convergence to the optimal point within $T_C = 0.495 \,\text{sec}$.}
%     \label{fig: Comp_disturbance}
% \end{figure}

\begin{figure}[h!]
\centering

\begin{subfigure}[b]{0.48\linewidth}
\centering
\includegraphics[width=1.3\linewidth]{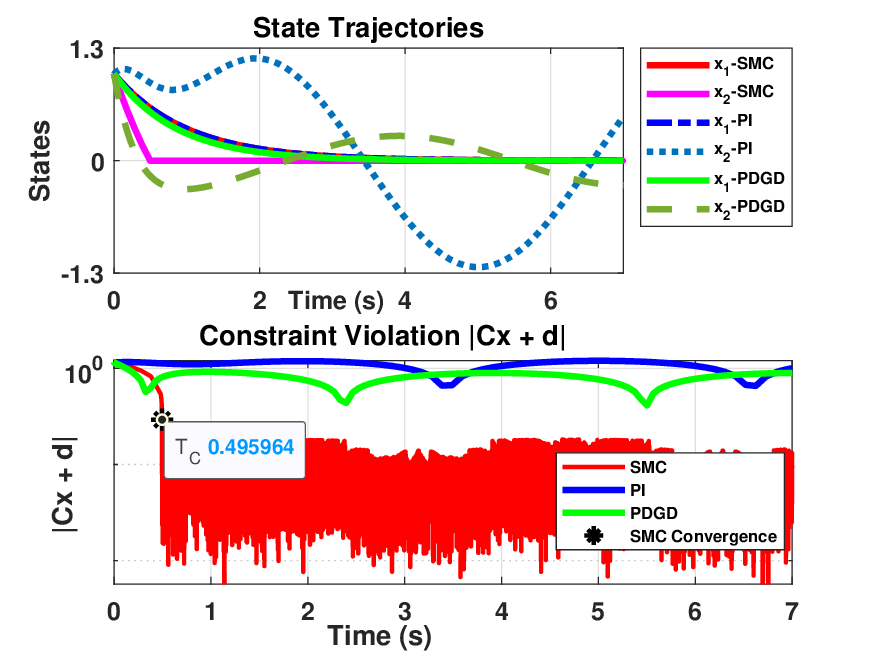}
\caption{}
\label{fig:comp_disturbance}
\end{subfigure}
\hfill
\begin{subfigure}[b]{0.4\linewidth}
\centering
\begin{tikzpicture}[scale=0.6, every node/.style={font=\tiny}]

% Axes
\draw[->, thin, gray] (-2.8,0) -- (3,0) node[below right] {$x$};
\draw[->, thin, gray] (0,-2.2) -- (0,2.5) node[above] {$\dot{x}$};

% Sliding manifold (vertical line)
\draw[thick, blue] (0.7,-2.2) -- (0.7,2.2);
\node[blue, rotate=90] at (0.82,0) {$h(x)=0$};

% Shaded regions
\fill[red!10] (-2.8,-2.2) -- (0.7,-2.2) -- (0.7,2.2) -- (-2.8,2.2) -- cycle;
\fill[green!10] (0.7,-2.2) -- (2.8,-2.2) -- (2.8,2.2) -- (0.7,2.2) -- cycle;

% Reaching trajectories
\draw[->, thick, red] (2.2,1.6) .. controls (1.5,1.0) and (1.1,0.5) .. (0.75,0.12);
\draw[->, thick, red] (2.2,-1.4) .. controls (1.5,-0.9) and (1.1,-0.4) .. (0.75,-0.1);
\draw[->, thick, red] (-1.5,1.8) .. controls (-0.3,1.0) and (0.4,0.5) .. (0.73,0.1);
\draw[->, thick, red] (-1.5,-1.6) .. controls (-0.3,-0.9) and (0.4,-0.4) .. (0.73,-0.08);

% Sliding motion
\draw[->, thick, green!60!black] (0.9,0) -- (2.5,0);
\draw[->, thick, green!60!black] (0.9,0.25) -- (2.1,0.25);
\draw[->, thick, green!60!black] (0.9,-0.25) -- (2.1,-0.25);

% KKT point
\fill[magenta!80!purple] (2.5,0) circle (2.5pt);
\node[below right, magenta!80!purple] at (2.5,0) {$x^*$};

% Phase labels
\node[red, font=\tiny] at (-1.2,1.4) {Reaching};
\node[green!60!black, font=\tiny] at (1.6,0.9) {Sliding};

% Time annotation
\node at (-1.8,-1.7) {$t<T$};
\node at (1.8,-1.7) {$t\ge T$};

% Sliding surface annotation
\node[blue] at (0.7,-1.6) {$\uparrow$};
\node[blue] at (0.7,-1.3) {$S=h=0$};

\end{tikzpicture}
\caption{}
\label{fig:smc_phase}
\end{subfigure}

\caption{(a) Validation of Theorem 6: Comparison of the proposed SMC-based constrained optimization method with PI-CMO \cite{cerone2025new} and PDGD\cite{arrow1958studies} under disturbance. While both PI-CMO and PDGD diverge in the presence of disturbances, the SMC-based approach achieves convergence to the optimal point within $T_C = 0.495 \,\text{sec}$. (b) Phase portrait: Trajectories (red) reach the sliding surface $S(x)=h(x)=0$ (vertical blue line) in finite time $T$, then slide along it (green arrows) toward the KKT point $x^*$ (magenta). The sliding surface is defined as the set $\{x \in \mathbb{R}^n \mid h(x)=0\}$.}
\label{fig:combined}
\end{figure}
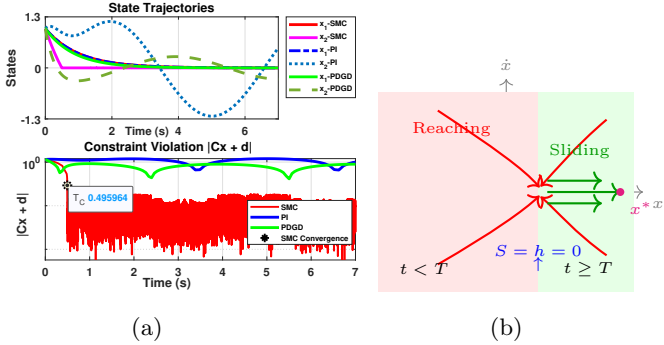

Moreover, the associated optimization problem converges to its optimal operating point. In contrast, both the primal-dual gradient descent method and the PI-CMO algorithm proposed in \cite{cerone2025new} exhibit sensitivity to the disturbance, resulting in significant deviation from the optimal trajectory and, in some cases, instability. Additionally conceptual diagram of the proposed theory is shown in Fig. \ref{fig:smc_phase}.  

%%%%%%%%%%%%%%%%%%%%%%%%%%%%%%%%%%%%%%%%%%%%%%%%%%%%%%%%%%%%%%%%%%%%%%%%%%%%%%%%%%%%%%%
\subsection{Obstacle Avoidance}
We illustrate the proposed sliding–mode based constrained optimization framework using a mobile robot navigation problem in a cluttered planar environment. The robot state is $\mathbf{q}=[x~y]^T\in\mathbb{R}^2$ with first–order dynamics $\dot{\mathbf{q}}=\mathbf{u}$. The task is to steer the system from the initial configuration $\mathbf{q}_0$ to the goal $\mathbf{q}_g$ while satisfying the distance–based safety constraints
% \begin{equation}
$h_i(\mathbf{q})=\|\mathbf{q}-\mathbf{q}_{\mathrm{obs}}^i\|-R_s= 0,
\quad 
\mathbf{q}_{\mathrm{obs}}^1=\begin{bmatrix}3\\1\end{bmatrix},~
\mathbf{q}_{\mathrm{obs}}^2=\begin{bmatrix}3\\-1\end{bmatrix}$,
% \end{equation}
with safety radius $R_s=0.8$. The objective function to be minimized is the quadratic potential
% \begin{equation}
$\phi(\mathbf{q})=\tfrac12\|\mathbf{q}-\mathbf{q}_g\|^2$.
% \end{equation}
For baseline comparison, we implement the classical artificial potential field (APF) controller \cite{szczepanski2023safe}
% \begin{equation}
$\mathbf{u}_{\mathrm{APF}}=-\nabla \Phi(\mathbf{q})$,
% \end{equation}
where $\Phi(\mathbf{q})=\Phi_{\mathrm{att}}(\mathbf{q})+\sum_i\Phi_{\mathrm{rep}}^i(\mathbf{q})$ combines a quadratic attractive term and inverse–distance repulsive potentials. Due to the symmetric obstacle configuration, the repulsive gradients balance along the corridor, and the robot becomes trapped at a non-goal equilibrium $\mathbf{q}^\ast\neq \mathbf{q}_g$ satisfying $\nabla\Phi(\mathbf{q}^\ast)=0$. This behavior confirms the well-known local-minima drawback of APF navigation. The projected gradient flow corresponding to the constrained optimization problem
\begin{equation}
\min_{\mathbf{q}}~ \phi(\mathbf{q})
\quad\text{s.t.} ~~~~
h_i(\mathbf{q})= 0
\end{equation}
is given by
% \begin{equation}
$\dot{\mathbf{q}}=-P(\mathbf{q})\nabla \phi(\mathbf{q}),
\quad
P(\mathbf{q})=I-J_h^T(J_hJ_h^T)^{-1}J_h$,
% \end{equation}
The projection operator preserves feasibility but yields only asymptotic convergence, with slow motion along curved obstacle boundaries where the projected descent direction becomes nearly tangent to the feasible set. 
\begin{figure}[h!]
    \centering
    %--- Subfigure 1: Trajectory Evolution ---
    \begin{subfigure}{0.9\linewidth}
        \centering
        \includegraphics[width=0.9\linewidth]{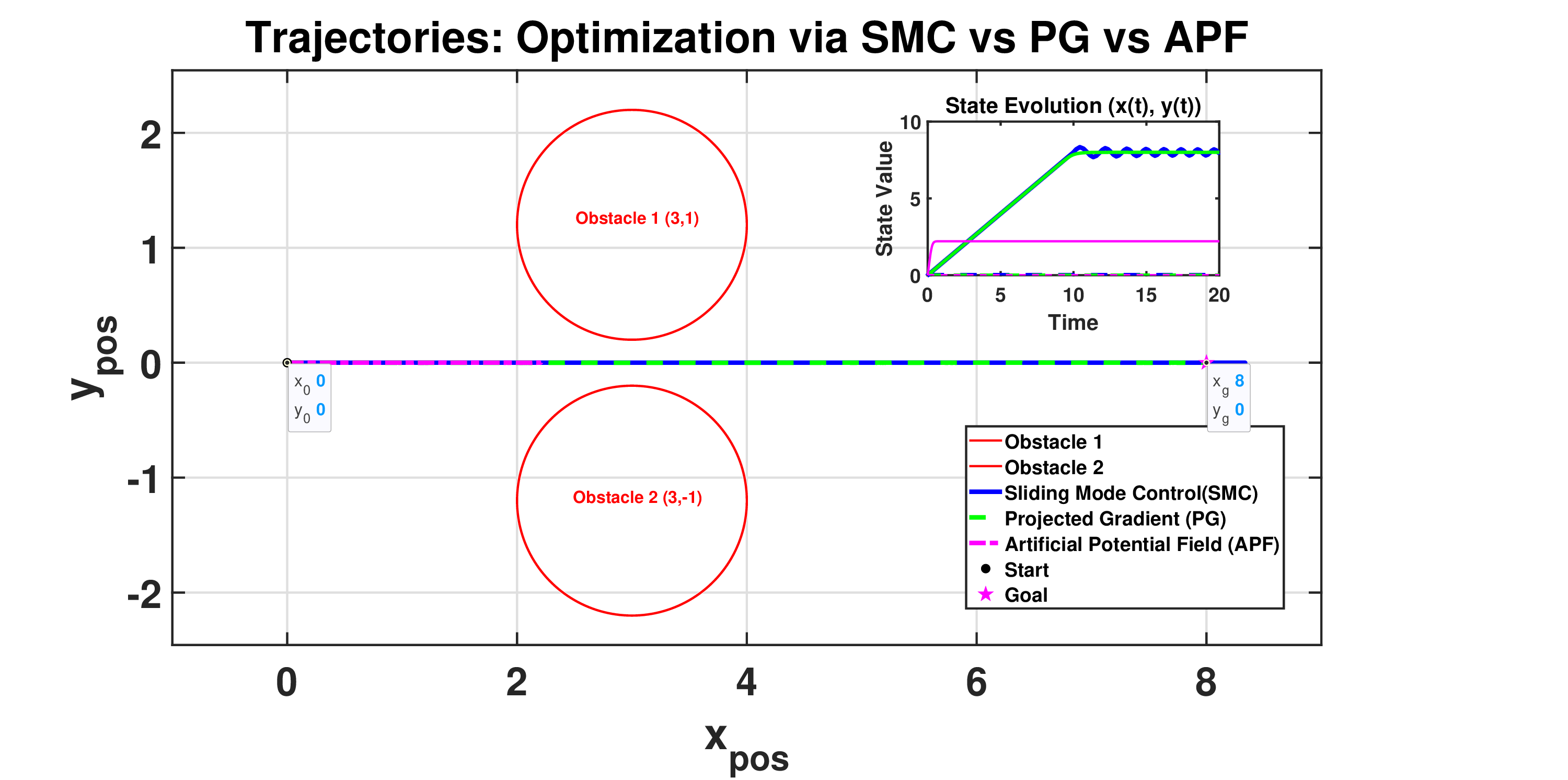}
        \caption{Trajectories Evolution}
        \label{fig:traj}
    \end{subfigure}
    %--- Subfigure 2: Constraint Satisfaction ---
    \begin{subfigure}{0.9\linewidth}
        \centering
        \includegraphics[width=0.9\linewidth]{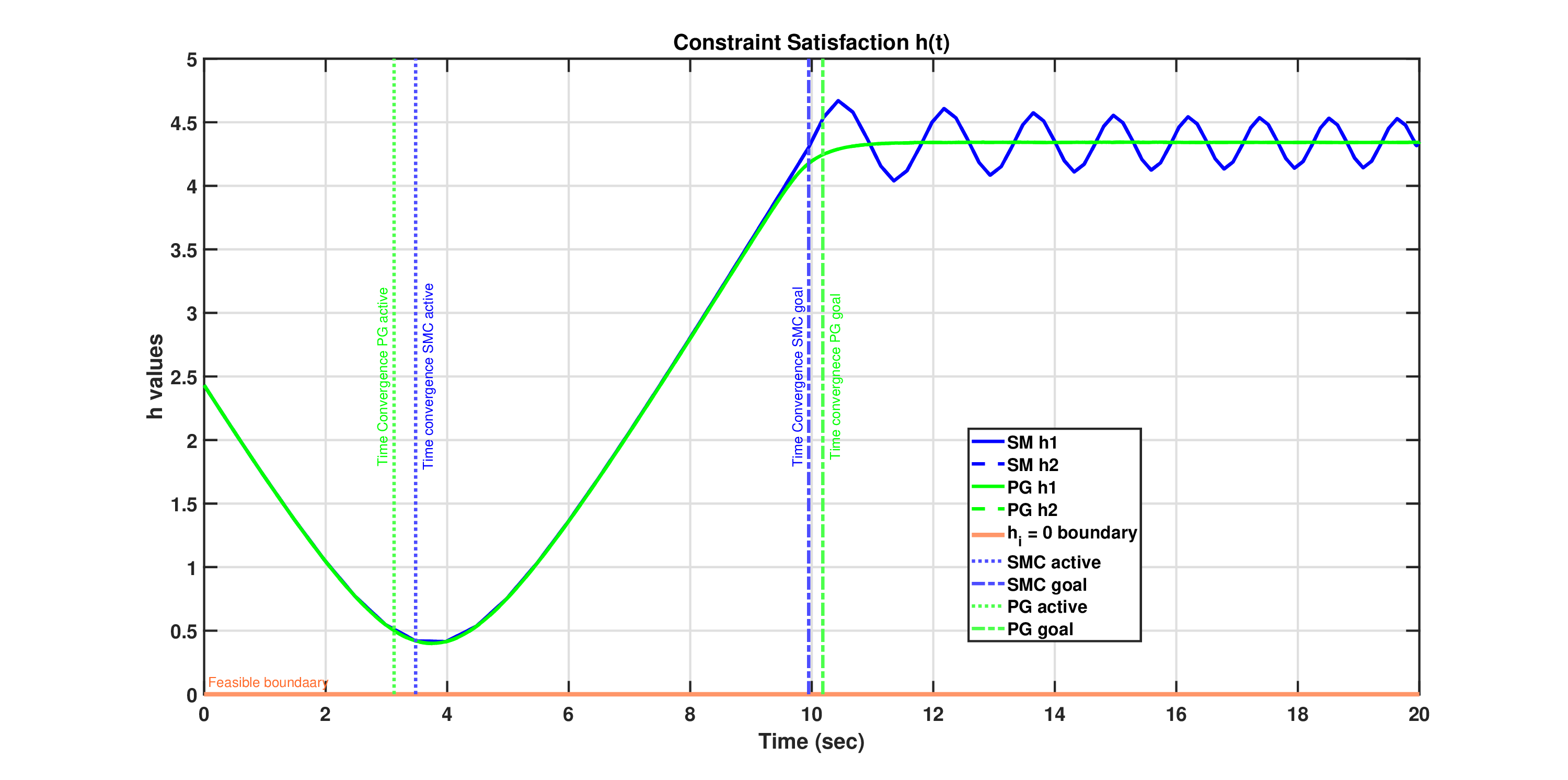}
        \caption{Constraint Satisfaction}
        \label{fig:constraint_satisfaction}
    \end{subfigure}
    \caption{Obstacle avoidance under the proposed SMC law \eqref{eq: perturbed system} in an elliptical narrow passage where APF fails with initial position $\mathbf{q}_0=[0;0]$, $\eta_{\max}=0.3$, and $K=20$. The Fig. \ref{fig:traj} represents a comparison of state trajectories implementing APF, PGF and the proposed model. APF becomes trapped in the narrow passage, whereas the proposed method successfully reaches the goal. Figure \ref {fig:constraint_satisfaction}  shows a comparison with projected gradient flow, faster convergence is achieved, with expected chattering due to the sliding-mode implementation.}
    \label{fig: Trajectory and constraint satisfication}
\end{figure}
The proposed sliding-mode constrained optimization augments the nominal gradient flow with a discontinuous correction, yielding the closed-loop dynamics \eqref{eq: perturbed system} and constraint evolution $\dot{h}=-K\,\mathrm{sgn}(h)+J_hJ_h^{\top}\eta$. Finite--time convergence to $h=0$ is ensured if $K>\sigma_{\max}(J_hJ_h^{\top})\bar{\eta}$, with settling time $T_{\mathrm{SMC}}\le\frac{\sqrt{2}\|h(0)\|}{K-\sigma_{\max}(J_hJ_h^{\top})\bar{\eta}}$. As shown in Fig.~\ref{fig: Trajectory and constraint satisfication}, APF fails due to symmetric repulsion and projected gradient flow converges slowly along obstacle boundaries, whereas the proposed method enforces $h_i(t)=0$ in finite time $0.41~s$, preserves feasibility, and reaches the goal, demonstrating finite-time safety, robustness, and elimination of APF-type local minima.
% % \begin{equation}
% \dot{\mathbf{q}} = 
% -\nabla \phi(\mathbf{q}) - J_h^\top \lambda+J_h^T\eta(t),
% % \end{equation}
%%%%%%%%%%%%%%%%%%%%%%%%%%%%%%%%%%%%%%%%%%%%%%%%%%%%%%%%%%%%%%%%%%%%%%%%%%%%%%%%%%%%%%%%%%%%%%%%%%%%%%
\subsection{Shidoku Puzzle}
Consider Shidoku is a 4x4 version with constraints expressed as 
\begin{equation}
\setlength{\arraycolsep}{1pt}
\begin{array}{rl}
\text{Integer Value:}   & \prod_{h=1}^4 (x_{ij} - h) = 0, \\
\text{Rows:}    & \sum_j x_{ij} = 10,\ \prod_j x_{ij} = 24, \\
\text{Cols:}    & \sum_i x_{ij} = 10,\ \prod_i x_{ij} = 24, \\
\text{Blocks:}  & \sum_{(k,\ell)\in B_p} x_{k\ell} = 10,\ \prod_{(k,\ell)\in B_p} x_{k\ell} = 24, \\
\text{Fixed position:}   & x_{1,2}=1,\ x_{1,4}=4,\ x_{3,1}=2,\ x_{3,4}=3.
\end{array}
\label{eq:shidoku_compact_constraints}
\end{equation}
The Shidoku puzzle is solved using the SMC-based optimization framework defined by the closed-loop dynamics
$
\dot{x} = -J^\top (J_h J_h^\top)^{-1} \left( K \cdot \frac{h}{|h| + \varepsilon} + \alpha h \right),
$
where $K = 5$, $\alpha = 0.1$, and $\varepsilon = 10^{-3}$ ensures smooth approximation of the discontinuous $\mathrm{sgn}$ function to reduce chattering. The linear term $\alpha h$ accelerates convergence toward the constraint manifold, while the signum-like term enforces constraint satisfaction robustly as shown in Fig. \ref{shiduko-combined}. 
\begin{figure}[h!]
    \centering
    \begin{subfigure}[t]{0.48\linewidth}
        \centering
        \includegraphics[width=\linewidth]{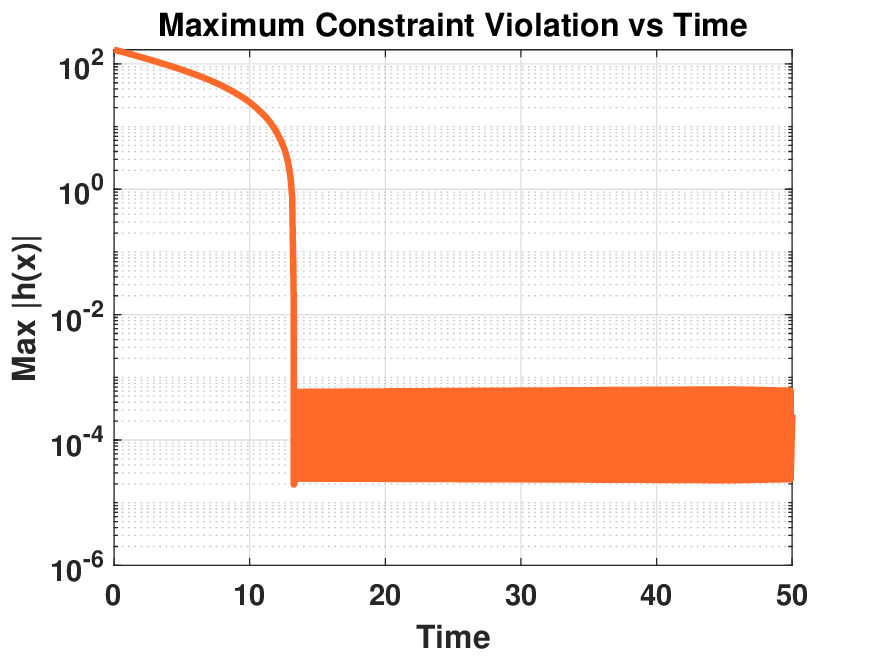}
        \caption{Constraint violation over time}
    \end{subfigure}
    \hfill
    \begin{subfigure}[t]{0.48\linewidth}
        \centering
        \includegraphics[width=\linewidth]{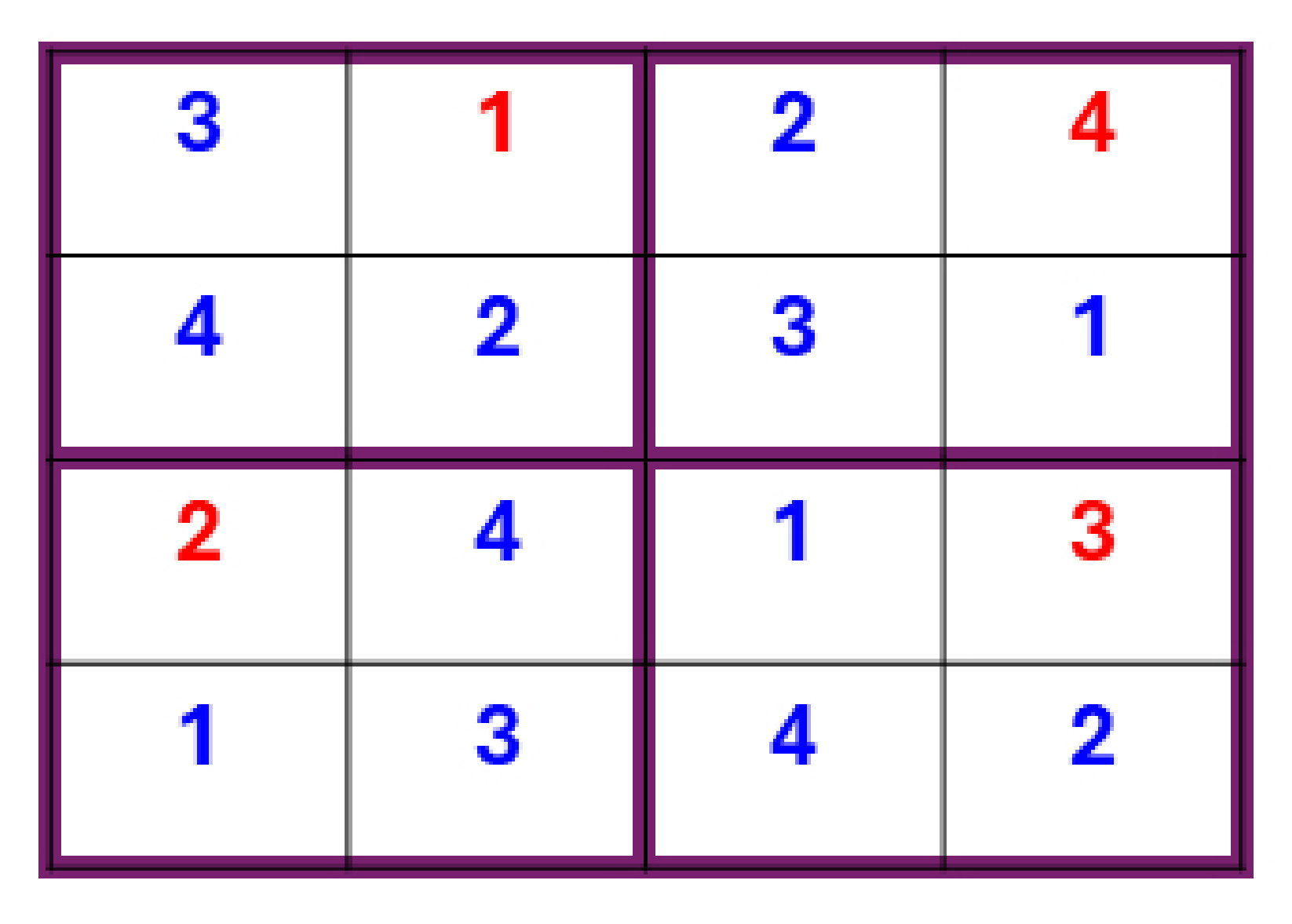}
        \caption{Block structure of the puzzle} 
    \end{subfigure}
    \caption{\textit{Shidoku puzzle visualization:} Fig.(a) shows constraint violations (\(||h(x(t))||_\infty\)) under sliding mode control, showing finite-time convergence with maximum constraint violation: $2.5291e-04$. Fig. (b) Partitioning of the 4×4 grid into four 2×2 blocks, each requiring unique permutations of $\{1,2,3,4\}$. The purple lines highlight block boundaries, with fixed values (\textcolor{red}{Red}) and variables (\textcolor{blue}{Blue}) evolving to satisfy row, column, and block constraints simultaneously.}
    \label{shiduko-combined}
\end{figure}
The initial condition is chosen as $x_0 = \texttt{rand(1,4)}$, corresponding to the 12 variable entries in the Shidoku grid. 
\begin{figure}[h!]
    \centering
    \includegraphics[width=\linewidth]{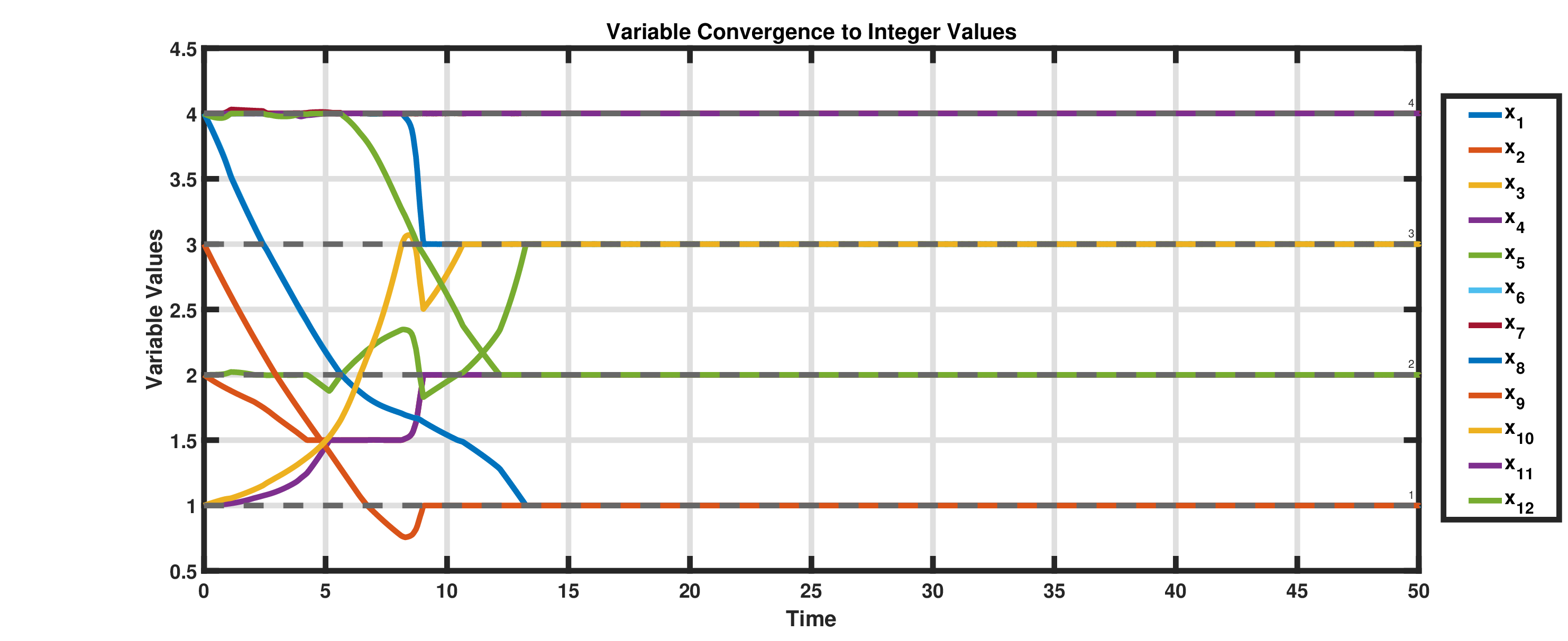}
    \caption{This figure shows the trajectories of all 12 variables $(x_1–x_{12})$ converging to integer solutions {1,2,3,4} under constrained dynamics, with different colours of solid lines marking target values.}
    \label{state-eva_puzzle}
\end{figure} 
The system is simulated using \textit{MATLAB ode45} solver for numerical integration. We perform 5 runs with different random initial conditions, and the solution always converges within a limited time interval $\max_{i} |h_i(x(t))| \leq \epsilon$, shown in Fig. \ref{state-eva_puzzle}.
\begin{rem}
While full row rank of $J_h(x)$ ensures unique equivalent control and KKT recovery, it is not necessary for finite-time feasibility. In the Shidoku problem, $J_h(x)$ is rank-deficient due to redundant constraints; however, finite-time convergence to $S(x)=\{x:h(x)=0\}$ is guaranteed as long as $J_h(x)^\top h(x)\neq 0$ for $h(x)\neq 0$, ensuring nontrivial sliding dynamics without requiring LICQ.
\end{rem}
% In the Shidoku formulation \eqref{eq:shidoku_compact_constraints}, $J_h(x)$ is typically rank-deficient due to redundant constraints. The proposed SMC framework does not require LICQ or full row rank and guarantees finite-time convergence to $S(x)=\{x:h(x)=0\}$ whenever $J_h(x)^{\top}h(x)\neq 0$ for $h(x)\neq 0$.

\subsection{Distributed Parameter Estimation}
Consider $N$ agents communicating over a connected undirected graph with Laplacian $L \in \mathbb{R}^{N \times N}$. Each agent $i$ collects local measurements $y_i = H_i \theta_i + v_i$, where $\theta_i \in \mathbb{R}^n$ is a local estimate of an unknown parameter $\theta^\star$, and $v_i$ is zero-mean Gaussian noise with covariance $R_i \succ 0$. Define the stacked variable
% \begin{equation}
$x = \operatorname{col}(\theta_1,\dots,\theta_N) \in \mathbb{R}^{Nn}$.
% \end{equation}
The distributed maximum likelihood estimation problem is formulated as
\begin{equation}
\begin{aligned}
\min_{x} \quad &
\phi(x) = \frac{1}{2}\sum_{i=1}^{N}
(y_i - H_i\theta_i)^\top R_i^{-1}(y_i - H_i\theta_i) \\
\text{s.t.} \quad &
h(x) = (L \otimes I_n)x = 0,
\end{aligned}
\label{eq:dpe}
\end{equation}
which enforces exact consensus among agents. The gradient and constraint Jacobian are $\nabla \phi(x)=\operatorname{col}(-H_i^\top R_i^{-1}(y_i - H_i\theta_i))$ and $J_h=L\otimes I_n$, respectively. To solve \eqref{eq:dpe}, we are using  \eqref{eq:ntsm_sys} with control input designed as \eqref{eq: ntsm control law}. Let $z_1 = h(x), z_2 = J_h \dot{x}$, and define the nonsingular terminal sliding variable $S = z_1 + \frac{1}{\beta}|z_2|^{\gamma}\mathrm{sgn}(z_2)$. Under this implementation, the consensus constraint $h(x)=0$ is satisfied in finite time, and all agents recover the common optimal estimate $\theta_1 = \cdots = \theta_N = \theta^\star$, where
$\theta^\star =(\sum_{i=1}^{N} H_i^\top R_i^{-1} H_i)^{-1}
(\sum_{i=1}^{N} H_i^\top R_i^{-1} y_i$). 
\begin{figure}[h!]
    \centering
    \includegraphics[width=\linewidth]{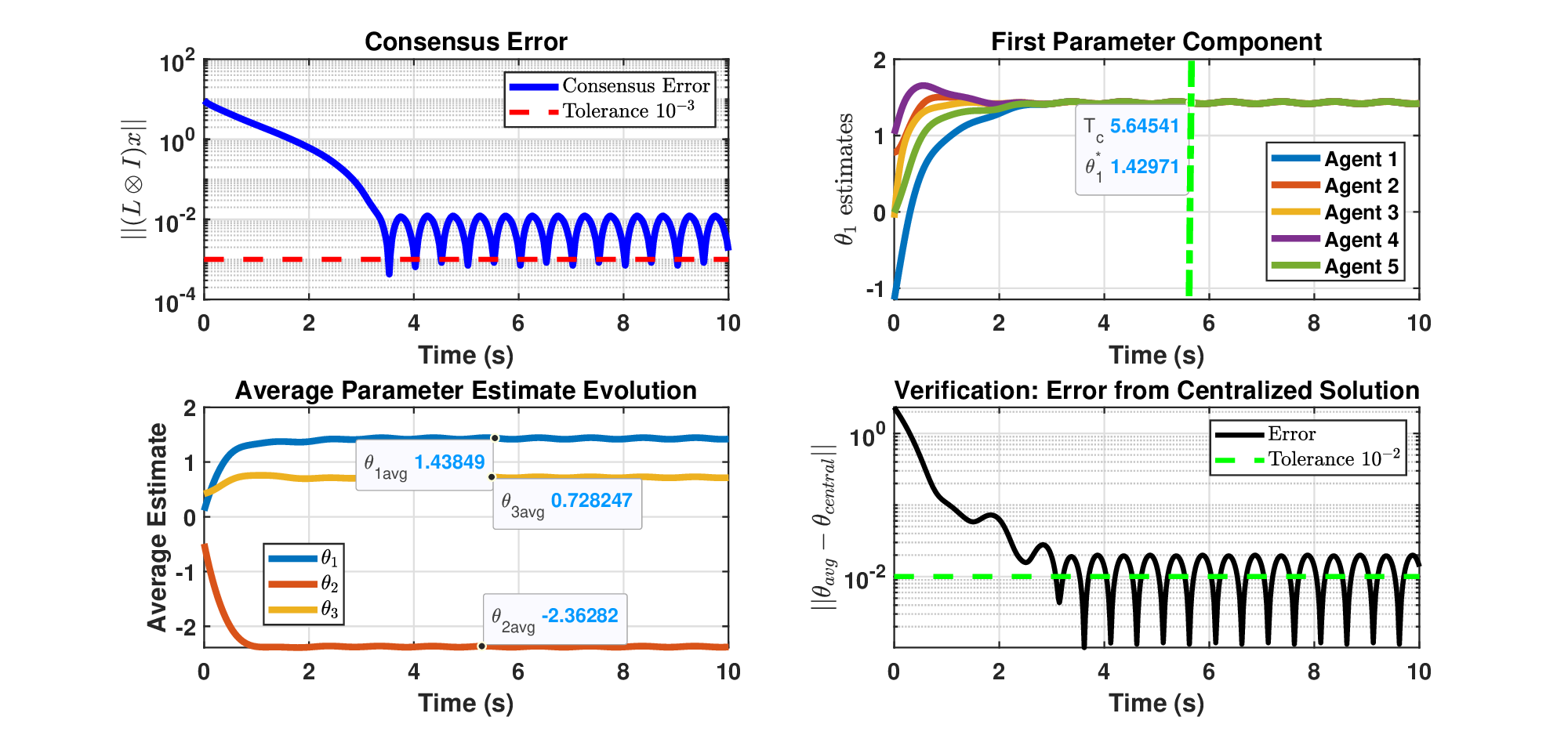}
    \caption{Convergence of all agent parameters to the centralized optimal value  $[1.4315,\,-2.3718,\,0.7213]$ with convergence time $T_c \le 5.5~\mathrm{s}$ in the presence of matched disturbance $\xi = 0.1\sin(2\pi t)\mathbf{1}_{Nn}$, for $N=5$, $n=3$, random initial conditions $x_{N(n)}=\mathrm{randn}(n,1)$, and $\lambda_0=0$.}
    \label{fig:parameter estimations}
\end{figure}
The Fig.~\ref{fig:parameter estimations} shows convergence of each node state to the true centralized value under the NTSMC law of Theorem~7, implemented with parameters $p=0.5$, $\gamma=1.5$, $\rho=0.7$, $\beta=2.0$, $K_1=5I$, $K_2=3I$, and $\eta=0.5$.

%%%%%%%%%%%%%%%%%%%%%%%%%%%%%%%%%%%%%%%%%%%%%%%%%%%%%%%%%%%%%%%%%%%%%%%%%%%%%
\section{Conclusion}
This paper proposed a sliding-mode control framework for continuous-time constrained optimization that guarantees finite-time satisfaction of equality constraints and robustness to matched disturbances and measurement noise. By embedding discontinuous SMC laws into Lagrangian-based dynamics, the optimality conditions were reformulated as a nonsmooth control-affine system, and Lyapunov analysis established finite-time reachability of the constraint manifold, independent of convexity, outperforming PDGD and PI-CMO on nonconvex problems. Finite-time convergence to optimal solutions was further achieved via a normalized gradient flow with NTSMC, demonstrated through distributed parameter estimation. In robot obstacle avoidance, the method outperformed APF (failure in narrow passages) and PGF (only asymptotic constraint satisfaction), while Shidoku feasibility results showed finite-time reachability without full-rank assumptions and recovery of KKT equivalence when the rank condition holds. 

Future work will extend the framework to a discrete-time setting, investigating whether key properties such as finite-time convergence and constraint enforcement are preserved under discretization, and developing suitable modifications when necessary. Additionally, co-authors are actively working on extending the framework to handle inequality constraints using augmented Lagrangian or penalty-based approaches. We will also explore distributed formulations to enable scalable implementations for large-scale systems.

%%%%%%%%%%%%%%%%%%%%%%%%%%%%%%%%%%%%%%%%%%%%%%%%%%%%%%%%%%%%%%%%%%%%%%%%%%%%%%%%%%%%%%%%%%%%%%%%%%%%%%%%%%%%%%%%%%%
 \section{Appendix} \label{STA}
 Since $\dot{h}=u$ (relative degree 1), the discontinuous control can be replaced by the Super-Twisting Algorithm (STA) to ensure continuous control with finite-time convergence. The STA is
\begin{equation}
    u = -K_1 |h|^{1/2}\operatorname{sgn}(h) + z,\quad \dot{z} = -K_2 \operatorname{sgn}(h),\quad K_1,K_2>0,
\end{equation}
leading to the second-order sliding dynamics $\dot{h} = -K_1 |h|^{1/2}\operatorname{sgn}(h) + z$, $\dot{z} = -K_2 \operatorname{sgn}(h)$. The control is continuous since the discontinuity is only in $\dot{z}$. Finite-time convergence to $h=0$, $\dot{h}=0$ holds if $K_2>0$, $K_1>2\sqrt{K_2}$; a practical choice is $K_1=1.5\sqrt{K_2}$. The Super-Twisting Algorithm offers an attractive alternative, preserving finite-time convergence of SMC while ensuring continuous control and eliminating chattering for relative degree 1 systems.
\begin{table}[h!]
\centering
% \caption{\textcolor{blue}{Comparison between Ideal SMC, Smooth SMC, and STA}}
\scriptsize
\setlength{\tabcolsep}{2.8pt}
\renewcommand{\arraystretch}{1.1}
\resizebox{\columnwidth}{!}{
\begin{tabular}{|c|c|c|c|c|}
\hline
\textbf{Method} & \textbf{Dynamics} & \textbf{FT} & \textbf{Continuity} & \textbf{Chattering} \\
\hline
Ideal SMC & $\dot{h} = -K\,\operatorname{sgn}(h)$ & Yes & Discontinuous & High \\
\hline
Smooth SMC & $\dot{h} = -K\,\operatorname{sat}(h/\epsilon)$ & No & Continuous & Low \\
\hline
STA & 
$\begin{aligned}
\dot{h} &= -K_1 |h|^{1/2}\operatorname{sgn}(h) \\
\dot{z} &= -K_2 \operatorname{sgn}(h)
\end{aligned}$ 
& Yes & Continuous & None \\
\hline
\end{tabular}
}

\label{tab:smc_comparison}
\end{table}

\bibliographystyle{ieeetr}        % Include this if you use BibTeX 
\bibliography{ref}        
% \appendix
% \section{A summary of Latin grammar}    % Each appendix must have a short title.
% \section{Some Latin vocabulary}         % Sections and subsections are supported  
                                        % in the appendices.
\end{document}